\documentclass[12pt]{article}
\usepackage{amsmath, amsfonts, amssymb}

\newtheorem{theorem}{Theorem}[section]
\newtheorem{lemma}[theorem]{Lemma}

\newtheorem{proposition}[theorem]{Proposition}

\newtheorem{remark}[theorem]{Remark}
\newtheorem{definition}[theorem]{Definition}

\newcommand{\cqd}{\hspace{10pt}\fbox{}}

\begin{document}
\title{Resonant elliptic problems under Cerami condition}
\author{
 {Edcarlos D. da Silva\thanks{}}\\
{\small Universidade Federal de Goi\'as}\\
{\small IME-UFG, Caixa Postal 131, CEP:
74001-970 , Goi\^ania - GO, Brazil}}

\date{}
\maketitle
\begin{abstract}
We establish existence and multiplicity of solutions to resonant elliptic problems using appropriate variational methods. In order to prove the compactness required in our main theorems we apply the well known Cerami condition.
\end{abstract}

$\mathbf{keywords}$
Elliptic problems, Resonance, Strong Resonance, Multiplicity, Variational Methods.\\
MSC Primary 35J20, Secondary 35J65.

\section{Introduction}

In this paper we establish existence and multiplicity of solutions to elliptic problem
\begin{equation}\label{pi}
       \text{ }
       \left\{ \begin{array}[c]{cc}
           - \Delta u  = \lambda_{1}  u + f(x,u) \, \, \mbox{in} \, \, \Omega, & \\
           u = 0 \, \, \mbox{on}\, \, \partial \Omega, &\\
        \end{array}
      \right.
      \end{equation}
where $\Omega \subseteq \mathbb{R}^{N}, N \geq 3,$ is a bounded domain with regular boundary, $\lambda_{1}$ denotes the first positive eigenvalue on $(- \Delta, H^{1}_{0}(\Omega))$ and $f \in C(\overline{\Omega} \times \mathbb{R},\mathbb{R})$ is continuous function satisfying the following limit:
\begin{equation}\label{lim}
      \lim_{|t| \rightarrow \infty} \dfrac{f(x,t)}{t} = 0,
      \end{equation}
uniformly and for all $x \in \Omega$.

From a standard variational point of view, finding solutions of \eqref{pi} in $H^{1}_{0}(\Omega)$ is equivalent to find the critical points of the $C^{1}$
functional $J: H^{1}_{0}(\Omega) \rightarrow \mathbb{R}$ given by
               \begin{equation}\label{4}
               J(u) = \dfrac{1}{2}\int_{\Omega}|\nabla u|^{2}dx - \dfrac{\lambda_{1}}{2} \int_{\Omega} u^{2}dx - \int_{\Omega} F(x,u)dx
               \end{equation}
where
      \begin{equation*}
      F(x,t) = \int_{0}^{t}f(x,s)ds,  \ \, \ x \in \Omega,\ t \in \mathbb{R}.
      \end{equation*}

In that case the problem \eqref{pi} becomes resonant at infinity which have been studied by many authors in recent years, see \cite{LL}, \cite{BBF}, \cite{BC}, \cite{B}, \cite{ALP}, \cite{S}, \cite{costa}, \cite{costa2}, \cite{Liu} and references therein.

It is worthwhile to mention that the problem \eqref{pi} becomes strong resonant when the following conditions
\begin{equation}\label{lin}
      \lim_{|t| \rightarrow \infty} f(x,t) = 0, \,\, \mbox{and} \,\, |F(x, t)| \leq C, (x, t)  \in \Omega \times \mathbb{R}.
      \end{equation}
hold uniformly and for all $x \in \Omega$ for some $C > 0$. These problems have been considered in several works where was used, for example, the well known nonquadraticity condition which can be written as
\begin{equation*}
(NQ)^{+} \,\, \lim_{|t| \rightarrow \infty} 2 F(x,t) - f(x,t)t = \infty,
      \end{equation*}
or
\begin{equation*}
(NQ)^{-}  \,\,\lim_{|t| \rightarrow \infty} 2 F(x,t) - f(x,t)t = -\infty,
      \end{equation*}
uniformly and for any $x \in \Omega$. However, to the best our knowledge, there are few results around the problem \eqref{pi} without the conditions $(NQ)^{+}$ and $(NQ)^{-}$. We refer the reader to important works \cite{BBF}, \cite{S}, \cite{Li}, \cite{CM}, \cite{AC}.

Actually, the main purpose of this paper is to introduce a specific strong resonant elliptic condition given by
\begin{flushleft}\label{hsr}
$(HSR)$ There is a function $\widehat{F} \in C(\Omega,\mathbb{R})$ such that
\end{flushleft}
\begin{equation*}
\lim_{|t| \rightarrow \infty} t f(x,t) = 0, \,\, \mbox{and} \,\, |F(x, t)| \leq \widehat{F}(x), \,(x, t)  \in \Omega \times \mathbb{R}.
\end{equation*}
uniformly and for any $x \in \Omega$.

As the function $F$ is bounded by a continuous function and  $t f(x,t)$ is also bounded we see easily that $(NQ)^{+}$ and $(NQ)^{-}$ does not work under the condition $(HSR)$. In other words, we always have that
$$2 F(x,t) - f(x,t)t $$
is bounded from above and below under the hypothesis $(HSR)$.
So we have a natural ask: Is there solution for the problem \eqref{pi} under the condition $(HSR)$? Clearly, this ask has a partial answer when the functional $J$ satisfies some compactness propriety like the well known Cerami condition; see \cite{cerami}.

We point out that when either
\begin{flushleft}
$(F0)^{-}$ There is $a \in C(\Omega, \mathbb{R})$ such that
\end{flushleft}

$$ \limsup_{t \rightarrow \infty} t f(x, t) \leq a(x) \preceq 0$$
or
\begin{flushleft}
$(F0)^{+}$ There is $b \in C(\Omega, \mathbb{R})$ such that
\end{flushleft}

$$ \liminf_{t \rightarrow \infty} t f(x, t) \geq b(x) \succeq 0$$
holds we have that $J$ satisfies the Cerami condition at any levels of energy; see \cite{dasilva}. Here and throughout this paper $a(x)\preceq 0$ means that $a(x) \leq 0$ in $\Omega$ with strict inequality holding for some subset $\widehat{\Omega} \subset \Omega$ with positive Lebesgue measure.

In this way, our condition (HSR) complements the research for elliptic strong resonance problems using the fact that $(F0)^{+}$ and $(F0)^{-}$ does not work under our condition. For example, in this paper we will prove that the function $f(t) = t e^{-t^{2}}$ have its functional $J$ with the following property:
\begin{center}
    $J$ satisfies the Cerami condition at level $c \in \mathbb{R}$ if only if $c = 0$.
\end{center}

In that case we will classify the all levels of energy for the functional $J$ where the Cerami condition work or does not work. Similary results for Palais-Smale property have been considered in  \cite{costa}, \cite{AC}.

Recall also that Variational Methods have been studied in several works in order to ensure that elliptic problems under resonant conditions
admit existence and/or multiplicity of solutions. The main tool in that case is to find a sufficient condition for the compactness involved in
these methods. More specifically, it well known that Palais-Smale condition or Cerami condition are sufficient tools in order to prove the Deformation Lemma which is crucial in variational methods.

In order to describe our main results we will always consider the strong resonant situation given by $(HSR)$. In that case we shall define the following auxiliary continuous functions $F^{+}$ and $F^{-}$ given by
$$F^{+}(x) = \lim_{t \rightarrow + \infty} F(x,t), F^{-}(x) =  \lim_{t \rightarrow -\infty} F(x,t), x \in \Omega,$$
where the limits just above hold uniformly and for any $x \in \Omega$.

Now we shall assume the following hypothesis:
\begin{flushleft}
(F1) There is $t^{\star} \in \mathbb{R}$ such that
\end{flushleft}
$$\int_{\Omega} F(x, t^{\star} \phi_{1}) dx  > \min \left(\int_{\Omega} F^{+}(x) dx, \int_{\Omega} F^{-}(x) dx \right) \geq  0.$$

In what follows we assume the new hypotheses $(HSR)$ introduced in this work. Thus, using the Ekeland's Variational Principle, we shall prove the following result

\begin{theorem}\label{1}
Suppose (HSR) and $(F1)$. Then the problem \eqref{pi} admits at least one weak solution.
\end{theorem}

Next we will assume always that
$$f(x,0) \equiv 0, x \in \Omega$$
holds.
Then $u = 0$ is a trivial solution for the problem \eqref{pi}.
In this way, the main purpose now is to ensure the existence and multiplicity of nontrivial solutions. Thus we shall consider some additional hypotheses:

\begin{flushleft}
(F2) There are $t^{\pm} \in \mathbb{R}$ such that
\end{flushleft}
$$\int_{\Omega} F(x, t^{\pm} \phi_{1}) dx > 0.$$

\begin{flushleft}
(F3) There are $ \delta > 0$ and $\alpha \in (0, \lambda_{1})$ such that
\end{flushleft}
$$\lim_{t \rightarrow 0}  \frac{f(x,t)}{t} \leq \alpha - \lambda_{1}.$$

In this way, combining the Ekeland's Variational Principle and Mountain Pass Theorem, we can prove the following multiplicity result

\begin{theorem}\label{2}
Suppose $(HSR), (F1)$. Then the solution obtained from Theorem \ref{1} is nontrivial. In addition, assuming $(F2)$ the problem \eqref{pi} admits at least two nontrivial solutions. Furthermore, assuming also $(F3)$, problem \eqref{pi} possesses at least three nontrivial solutions.
\end{theorem}

Next we will consider the following hypotheses:
\begin{flushleft}
(F4) There are $ \epsilon > 0, \delta > 0$ such that
$$\dfrac{1}{2}(\lambda_{k} - \lambda_{1})t^{2} \leq F(x,t) \leq \dfrac{1}{2}(\lambda_{k + 1} - \lambda_{1} - \epsilon)t^{2}, x \in \Omega, |t| \leq \delta.$$
\end{flushleft}

\begin{flushleft}
(F5) There are continuous functions $a, b \in C(\Omega, \mathbb{R})$ satisfying either
$$f(x, \pm \phi_{1}) \geq a(x) \succ 0$$
or
$$f(x, \pm \phi_{1}) \leq b(x) \prec 0.$$
\end{flushleft}

Now we will prove a multiplicity result for our problem using the well known Liking Theorem provided by Brezis-Nirenberg, see \cite{BN}.
\begin{theorem}\label{3}
Suppose $(HSR), (F1), (F4), (F5)$. Then the problem \eqref{pi} possesses at least two nontrivial solutions.
\end{theorem}

Throughout this paper we use the following notations:
\begin{itemize}
  \item $\|\|$ denotes the norm in $H^{1}_{0}(\Omega),$
  \item $\|\|_{p}$ is the norm in $L^{p}(\Omega), 1 \leq p \leq 2^{\star},$
  \item $\phi_{1}$ denotes the first eigenfunction for $(- \Delta, H^{1}_{0}(\Omega))$,
  \item $C, C_{1}, \ldots$ represent positive real numbers.
\end{itemize}

This paper is organized as follows: In section 2 we give some preliminares results around our problem. In section 3 we prove our main theorems. Section 4 is devoted to simple examples using our condition $(HSR)$.
\section{Preliminares}

In this section we will prove the Cerami condition for specials levels of energy. After that, we shall considered some useful results in order to use min-max theorems.

Let $H$ be a Hilbert space. We recall that a functional $J : H \rightarrow \mathbb{R},$ of class $C^{1}$, satisfies the Cerami condtion at level $c
\in
\mathbb{R}$, in short $(Ce)_{c}$,  if for any sequences $(u_{n})_{n \in \mathbb{N}} \in H$ such that
$$J(u_{n}) \rightarrow c, \|J^{\prime}(u_{n})\|(1 + \|u_{n}\|) \rightarrow 0, \,\,\mbox{as} \,\, n \rightarrow \infty$$
admits a convergent subsequence. When $J$ satisfies the $(Ce)_{c}$ property for any $c \in \mathbb{R}$ we say purely that
$J$ satisfies the $(Ce)$ property.

Now we will consider a Hardy-Sobolev-Polya inequality provided in \cite{polya}. This inequality is a powerful tool in order to prove the Cerami condition for our problem.
\begin{proposition}\label{hardy}
Let $\phi \in H^{1}_{0}(\Omega)$ be a function. Then we obtain the following assertions:
\begin{description}
  \item[i)] $\dfrac{\phi}{\phi_{1}^{\tau}} \in L^{p}(\Omega)$ where $\tau \in [0,1]$ and
  $$\dfrac{1}{p} = \dfrac{1}{2} - (1 - \tau)\dfrac{1}{N}.$$
  \item[ii)] There is a constant $C > 0$ such that $$\int_{\Omega} \left|\dfrac{\phi}{\phi_{1}^{\tau}} \right|^{p}dx \leq C \|\phi\|^{p}, \forall \,\,\phi \in H^{1}_{0}(\Omega).$$
\end{description}
\end{proposition}

Next we will prove the $(Ce)_{c}$ property for some leves of energy $c \in \mathbb{R}$ in order to ensure the compactness required in the proof to our main theorems. More precisely, we can prove the following result:

\begin{lemma}\label{cerami}
Suppose $(HSR)$. Then the functional $J$ satisfies the $(Ce)_{c}$ condition if only if $c \in \mathbb{R}\backslash \Gamma$ where we define
$$\Gamma = \left\{- \int_{\Omega} F^{+}(x)dx, - \int_{\Omega} F^{-}(x)dx \right\}.$$
\end{lemma}

\begin{proof}
We will divide the proof into two steps.

$\mathbf{Step 1}$
First of all, we shall prove that $J$ satisfies the $(Ce)_{c}$ condition for any $c \in \mathbb{R}\backslash \Gamma$. Assume, by contradiction, that there exist a sequence $(u_{n})_{n \in \mathbb{N}}$ satisfying the following conditions:
\begin{itemize}
    \item $J(u_{n}) \rightarrow c,$
    \item $\|J^{\prime}(u_{n})\|(1 + \|u_{n}\|) \rightarrow 0,$
    \item $\|u_{n}\| \rightarrow \infty$ as $ n \rightarrow \infty.$
\end{itemize}
Define $v_{n} = \dfrac{u_{n}}{\|u_{n}\|}$. It follows that $v_{n}$ is bounded in $H^{1}_{0}(\Omega)$ and there exist $v \in H^{1}_{0}(\Omega)$
satisfying
\begin{itemize}
    \item $v_{n} \rightharpoonup v$  in $H^{1}_{0}(\Omega), v_{n} \rightarrow v$ in
    $L^{q}(\Omega), q \in [1, 2^{\star}),$
    \item $v_{n}(x) \rightarrow v(x)$ a.e. in
    $\Omega,$
    \item $ |v_{n}(x)| \leq h(x),$ for some $h \in L^{q}(\Omega)$.
\end{itemize}

On the other hand, we recall that
\begin{eqnarray}\label{13}
  \dfrac{J^{'}(u_{n})\phi}{\|u_{n}\|}  &=& \int_{\Omega} \nabla v_{n} \nabla \phi dx - \lambda_{1} \int_{\Omega}
 v_{n} \phi dx \nonumber - \int_{\Omega} \dfrac{f(x,u_{n})}{\|u_{n}\|} \phi dx \rightarrow 0 \nonumber \\
\end{eqnarray}
as $n \rightarrow \infty$ for any $\phi \in H^{1}_{0}(\Omega)$.

In that case, using the last equation, we conclude that
\begin{equation}
\int_{\Omega} \nabla v \nabla \phi dx = \lambda_{1} \int_{\Omega} v \phi dx, \phi \in H^{1}_{0}(\Omega).
\end{equation}
These facts show that $\|v\| \leq 1$. Besides that, using $v_{n}$ as test function in \eqref{13}, we see easily see that
\begin{equation*}
  \|v\|^{2} = \int_{\Omega} |\nabla v|^{2}dx \geq \lambda_{1} \int_{\Omega} v^{2} = \lim_{n\rightarrow \infty}\lambda_{1} \int_{\Omega} v_{n}^{2} = \lim_{n\rightarrow \infty} \int_{\Omega} |\nabla v_{n}|^{2} = 1
\end{equation*}
In other words, $v$ is an eigenfunction associated to $\lambda_{1}$ such that $\|v\| = 1$. In particular, it follows that $v = \pm \phi_{1}$.
As a consequence $v_{n} \rightarrow v$ in $H^{1}_{0}(\Omega)$. Therefore
$$u_{n}(x) \rightarrow + \infty, n \rightarrow \infty, x \in \Omega$$
or
$$u_{n}(x) \rightarrow - \infty, n \rightarrow \infty, x \in \Omega$$
Note that $(HSR)$ implies that $f(x,t)t$ is a bounded function in $\Omega \times \mathbb{R}$. Then
\begin{equation*}
\int_{\Omega} f(x, u_{n})u_{n} dx \rightarrow 0 , n \rightarrow \infty,
\end{equation*}
by a straightforward application of Lebesgue Convergence Theorem.

On the other hand, choosing $u_{n}$ as test function, we obtain
\begin{equation*}
\langle J^{'}(u_{n}), u_{n}\rangle= \int_{\Omega} |\nabla u_{n}|^{2}dx - \lambda_{1} \int_{\Omega} u_{n}^{2}dx - \int_{\Omega} f(x,u_{n})u_{n}dx.
\end{equation*}
Under these assumptions follows that
\begin{equation}
\lim_{n \rightarrow \infty}\int_{\Omega} |\nabla u_{n}|^{2}dx - \lambda_{1} \int_{\Omega} u_{n}^{2}dx = 0
\end{equation}

However, using the functional $J$, we have
\begin{equation*}
c = \lim_{n \rightarrow \infty} J(u_{n}) = - \lim_{n \rightarrow \infty} \int_{\Omega} F(x,u_{n})dx.
\end{equation*}
Using one more time Lesbesgue Convergence Theorem follows that
$$c = - \int_{\Omega} F^{+}(x)dx \,\,\mbox{or} \,\,c  = - \int_{\Omega}F^{-}(x)dx.$$
This is a contradiction! So we finish the proof for the first step.

$\mathbf{Step 2}$ In this step we shall prove that $J$ does not satisfy the $(Ce)_{c}$ condition for any number in $\Gamma$. Define $u_{n} = n \phi_{1}$ which is an unbounded sequence in $H^{1}_{0}(\Omega)$. We easily see that

\begin{equation*}
J(u_{n}) \rightarrow c
\end{equation*}
for some $c \in \Gamma$.
In addition we may prove that
\begin{equation}\label{8}
\|J^{'}(u_{n})\| \left( 1 + \|u_{n}\| \right) \rightarrow 0, n \rightarrow \infty.
\end{equation}
In order to do that we choose $\Omega_{n} \subset \subset \Omega$ satisfying
\begin{equation}\label{7}
|\Omega \backslash \Omega_{n}|^{\frac{1}{2}}  \leq \dfrac{\epsilon}{\|u_{n}\|}
\end{equation}
where $\epsilon > 0$.
So we take $n_{0} \in \mathbb{N}$ such that
\begin{equation*}
n \geq n_{0} \,\,\mbox{implies that} \,\, |f(x, u_{n})u_{n}| \leq \epsilon, x \in \Omega_{n}.
\end{equation*}

In this way, taking $\phi \in H^{1}_{0}(\Omega), v_{n} = \dfrac{u_{n}}{\|u_{n}\|}$, follows the following estimates
\begin{eqnarray}
\langle J^{'}(u_{n}), \phi \rangle \|u_{n}\|& \leq & \left|  \int_{\Omega} f(x, u_{n}) u_{n} v_{n}^{-1}\phi dx \right|  \leq \int_{\Omega} |f(x, u_{n}) u_{n}| |v_{n}|^{-1}  |\phi| dx  \nonumber \\
   &=& \int_{\Omega_{n}} |f(x, u_{n}) u_{n}| |v_{n}|^{-1} |\phi| dx +  \|u_{n}\| \int_{\Omega\backslash \Omega_{n}} |f(x, u_{n})| |\phi| dx \nonumber \\
\end{eqnarray}
Next we will analyze the terms on the right hand described just above.
We easily see that
\begin{equation}\label{6}
\int_{\Omega_{n}} |f(x, u_{n}) u_{n}| |v_{n}|^{-1} |\phi| dx \leq \epsilon \|\phi_{1}\| \int_{\Omega} \left|\dfrac{\phi}{\phi_{1}} \right| dx \leq \epsilon C \|\phi\| \|\phi_{1}\|
\end{equation}
where was used the Hardy-Littlewood-Polya inequality with $\tau = 1$; see  Proposition \ref{hardy}.

We also have that
\begin{eqnarray}\label{5}
\|u_{n}\| \int_{\Omega\backslash \Omega_{n}} |f(x, u_{n})| |\phi| dx &\leq& C \|u_{n}\| \int_{\Omega\backslash \Omega_{n}} |\phi| dx  \leq C \|u_{n}\| \left(\int_{\Omega\backslash \Omega_{n}} |\phi|^{2} dx \right)^{\frac{1}{2}} |\Omega\backslash \Omega_{n}|^{\frac{1}{2}} \nonumber \\
   &\leq& C \|u_{n}\| |\Omega \backslash \Omega_{n}|^{\frac{1}{2}}\|\phi\| \leq \epsilon C \|\phi\| \nonumber \\
\end{eqnarray}
where was used the fact that $f$ is bounded function, Holder's inequality, Sobolev Embedding and \eqref{7}. As a consequence \eqref{6} and \eqref{5} imply that
\begin{equation*}
\|J^{'}(u_{n})\| \|u_{n}\| \leq \epsilon C  + \epsilon C  \|\phi_{1}\|.
\end{equation*}
These inequalities imply \eqref{8} and the proof of this lemma is now complete. \cqd
\end{proof}

Next we consider a result which ensures that $J$ is bounded from below. This permit us to aplly the well known Ekeland's Variational Principle for our problem.
\begin{proposition}\label{p1}
Suppose $(HSR)$ and $(F1)$. Then $J$ satisfies
$$\inf\left\{J(u), u \in H^{1}_{0}(\Omega) \right\} > - C |\Omega|,$$
for some $C > 0$.
In particular, $J$ is bounded from below.
\end{proposition}
\begin{proof}
The proof of this proposition is a straightforward application of $(HSR)$. In that case we will omit the details.
\cqd
\end{proof}

In order to consider existence or multiplicity of solutions for our problem \eqref{pi} we shall prove a result that
ensures the well known mountain pass geometry. More specifically, we can prove
\begin{proposition}\label{p2}
Suppose $(HSR), (F1), (F3)$. Then $J$ has the following mountain pass geometry:
\begin{description}
  \item[a)] There are $\rho > 0$ and $\beta > 0$ such that
  $$J(u) \geq \beta, \,\,\mbox{for any} \,\, u \in H^{1}_{0}(\Omega),\|u\| = \rho.$$
  \item[b)] $J(t^{\star} \phi_{1}) < 0.$
\end{description}
\end{proposition}
\begin{proof}
Using $(F3)$ we can choose $C > 0$ such that
$$F(x,t) \leq \dfrac{1}{2} (\alpha - \lambda_{1})t^{2} + C|t|^{p}, (x, t ) \in \Omega \times \mathbb{R}.$$
This fact and Sobolev's Embedding yield
\begin{eqnarray}
  J(u) &\geq& \dfrac{1}{2}(1 - \dfrac{\alpha}{\lambda_{1}}) \|u\|^{2} - C\|u\|_{p}^{p} \geq \dfrac{1}{2}(1 - \dfrac{\alpha}{\lambda_{1}}) \|u\|^{2} - C\|u\|^{p} \nonumber\\
   &=&   \|u\|^{2} \left\{\dfrac{1}{2}(1 - \dfrac{\alpha}{\lambda_{1}}) - C\|u\|^{p - 2}  \right\} \geq \dfrac{1}{4}(1 - \dfrac{\alpha}{\lambda_{1}}) \|u\|^{2} > 0, \nonumber\\
\end{eqnarray}
for any $\|u\| \leq \rho, u \in H^{1}_{0}(\Omega)$ where $\rho > 0$ is small enough.

On the other hand, by $(F1)$, we see easily that
$$J(t^{\star} \phi_{1}) < 0.$$
This finishes the proof of this proposition
\cqd
\end{proof}

\begin{remark}\label{re1}
Using the hypotheses $(F2), (F3)$ and the same ideas discussed in the previous proposition it follows that $J^{\pm}$ defined by
$$J^{\pm}(u) = \dfrac{1}{2}\int_{\Omega} |\nabla u|^{2}dx - \dfrac{\lambda_{1}}{2}\int_{\Omega} u^{2}dx - \int_{\Omega} F^{\pm}(x,u)dx,$$
\end{remark}
have the mountain pass geometry. Here we define
$$F^{\pm}(x,t) = \int_{0}^{t} f^{\pm}(x,s)ds, (x, t)  \in \Omega \times \mathbb{R},$$
\begin{equation*}
       \text{ }
       f^{+}(x,t) =
       \left\{ \begin{array}[c]{cc}
        f(x,t), t \geq 0& \\
           0, t \leq 0, &\\
        \end{array}
      \right.
      \end{equation*}
and
\begin{equation*}
       \text{ }
       f^{-}(x,t) =
       \left\{ \begin{array}[c]{cc}
        f(x,t), t \leq 0& \\
           0, t \geq 0. &\\
        \end{array}
      \right.
      \end{equation*}

Next we shall consider a powerful result that implies that $J$ has a liking geometry. Let $c \in \mathbb{R}$ be such that $c \geq \inf J$.
We mention that our problem $J$ satisfies the well known $(Ce)_{c}$ condition for any $c \in \mathbb{R}\backslash \Gamma$ that shows the following useful result:

\begin{proposition}\label{pp}
Suppose $(F1)$. Let $J: H^{1}_{0}(\Omega) \rightarrow \mathbb{R}$ be the energy functional given by \eqref{4} that satisfies the $(Ce)_{c}$ condition for any $c \in \mathbb{R}\backslash \Gamma$
and $\inf J < 0$. Also assume that $u_{0}$ is a minimizer of  J and  $\{u_{0}, 0\}$ are the only critical
points.
Then for any neighborhood $U$ of $u_{0}$ and any  $\delta > 0$ such that $U \cap B_{\delta} = \emptyset $, we
can find $\zeta > 0$ satisfying
\begin{equation}
(1 + \|u\|) \| J^{'}(u)\| \geq \zeta, \forall \,\, u \in J^{c} \backslash ( U \cup B_{\delta}),
\end{equation}
where we define $J^{c} = \{ u \in H^{1}_{0}(\Omega):  J(u) \leq c \}, c > 0$ and $B_{\delta} = \{ u \in H^{1}_{0}(\Omega) : \|u\| <  \delta\}$.
\end{proposition}
\begin{proof}
The proof follows arguing by contradiction and using the Cerami condition for appropriate levels. Firstly, we consider the following case:

$\mathbf{Case 1}$ Arguing by contradiction, we obtain a sequence $u_{n} \in J^{c}\backslash U \cup B_{\delta}$ such that
\begin{description}
  \item[i)] $(1 + \|u_{n}\|)\|J^{'}(u_{n})\| \leq \dfrac{1}{n}$
  \item[ii)] $J(u_{n}) \leq c.$
\end{description}
In that case, we will be considered when
\begin{equation}\label{11}
|J(u_{n}) - r | \geq \epsilon, \, \forall \, r  \in \Gamma,
\end{equation}
holds for some $\epsilon > 0$. Here we remember that $\Gamma : = \{ c \in \mathbb{R} : J$ does not satisfy $(Ce)_{c}$ $\}$.

It is worthwhile to infer that $J^{c}\backslash U \cup B_{\delta}$ is closed.
Up to a subsequence, by $(Ce)_{c}$ condition for any $c \in \mathbb{R}\backslash \Gamma$, we can find $u \in J^{c}\backslash U \cup B_{\delta}$ satisfying $u_{n} \rightarrow u$ in $H^{1}_{0}(\Omega)$. Note that \eqref{11} is crucial in this case. As a consequence we have that
$$J^{'}(u) = 0, J(u) \in \mathbb{R}\backslash \Gamma.$$
But $\{ 0, u_{0} \}$ are the only critical points of $J$. Then $u = u_{0} \in U \subset U \cup B_{\delta}$ or $u = 0 \in B_{\delta} \subset U \cup B_{\delta}$. However $u \in J^{c}\backslash U \cup B_{\delta}$ which is a contradiction.

$\mathbf{Case 2}$ Now we will analyze the complementary case, i.e., when the sequence $(u_{n}) \in H^{1}_{0}(\Omega)$ discussed in the previous case satisfies
\begin{equation}
\lim_{n \rightarrow \infty} J(u_{n}) =  r \, \, \mbox{for some} \,\,r \in \Gamma.
\end{equation}

In this case, following the same ideas discussed above, we remember also that $u_{n} \in J^{c}\backslash U \cup B_{\delta}$
verifies the following conditions
\begin{description}
  \item[i)] $(1 + \|u_{n}\|)\|J^{'}(u_{n})\| \leq \dfrac{1}{n}$
  \item[ii)] $|J(u_{n}) - r| \leq \dfrac{1}{n}, r \in \Gamma.$
\end{description}
Now we will divide the proof into two steps. First, we will assume that $(u_{n})$ is a unbounded sequence.
Define $v_{n} = \dfrac{u_{n}}{\|u_{n}\|}$. Therefore $(v_{n})$ is bounded in $H^{1}_{0}(\Omega)$ and we can find $v \in H^{1}_{0}(\Omega)$
such that $v_{n} \rightharpoonup v$ in $H^{1}_{0}(\Omega)$.

On the other hand, we see easily that
\begin{eqnarray*}
  \dfrac{\langle J^{'}(u_{n}), \phi \rangle}{\|u_{n}\|} &=& \langle v_{n}, \phi \rangle - \lambda_{1} \int_{\Omega} v_{n} \phi dx - \int_{\Omega} \dfrac{f(x, u_{n})}{\|u_{n}\|} \phi dx   \nonumber \\
  &=&  \langle v_{n}, \phi \rangle - \lambda_{1} \int_{\Omega} v_{n} \phi dx - \int_{\Omega} \dfrac{f(x, u_{n})}{u_{n}} v_{n} \phi dx. \nonumber \\
\end{eqnarray*}
Doing $n \rightarrow \infty$ follows that
\begin{equation}\label{9}
\langle v, \phi \rangle - \lambda_{1} \int_{\Omega} v \phi dx = 0, \phi \in H^{1}_{0}(\Omega).
\end{equation}
In addition, using $i)$ and variational inequalities, follows that
\begin{equation}
\|v\|^{2} \geq \lambda_{1} \int_{\Omega} v^{2} dx = \lim_{n \rightarrow \infty} \left\{\|v_{n}\|^{2} - \int_{\Omega} \dfrac{f(x,u_{n})}{u_{n}} v_{n}^{2} dx- \dfrac{\langle J^{'}(u_{n}), u_{n} \rangle}{\|u_{n}\|^{2}}\right\} = 1
\end{equation}
where was used the fact that $(v_{n})$ is normalized in $H^{1}_{0}(\Omega)$. In particular, by weak convergence and previous inequalities, follows that $\|v\| = 1$ and $v = \pm \phi_{1}$. Here we take $\phi_{1}$ normalized on $H^{1}_{0}(\Omega)$.
This together with \eqref{9} shows that $v = u_{0}$ where  $\inf J = J(u_{0}) < 0$.

Under this hypotheses it follows that
\begin{equation*}
\langle J^{'}(v), \phi\rangle = \langle v, \phi \rangle - \lambda_{1}\int_{\Omega} v \phi dx - \int_{\Omega} f(x,v) \phi dx = 0, \phi \in H^{1}_{0}(\Omega).
\end{equation*}
Moreover, using \eqref{9} , we see easily that
\begin{equation}\label{10}
\langle J^{'}(v), \phi\rangle = - \int_{\Omega} f(x,v) \phi dx = 0, \phi \in H^{1}_{0}(\Omega).
\end{equation}
Now, by $(F5)$, we see that
\begin{equation*}
\int_{\Omega} f(x, \pm \phi_{1}) \phi_{1} dx \geq \int_{\Omega} a(x) \phi_{1} dx > 0
\end{equation*}
or
\begin{equation*}
\int_{\Omega} f(x, \pm \phi_{1}) \phi_{1} dx \leq \int_{\Omega} b(x) \phi_{1} dx < 0.
\end{equation*}
This is a contradiction with \eqref{10} using $\phi = \phi_{1}$. Thus $(u_{n})$ is a bounded sequence in $H^{1}_{0}(\Omega)$.

Next we will analyze the case when $(u_{n}) \in J^{c}\backslash U \cup B_{\delta} $ is a bounded sequence. In that case we always guarantee a point $u \in J^{c}\backslash U \cup B_{\delta}$ satisfying $u_{n} \rightharpoonup u$ in $H^{1}_{0}(\Omega)$. In this way using that $J^{'}$ has the form identity minus a compact operator it follows that
\begin{equation*}
\langle J^{'}(u), \phi \rangle = 0, \phi \in H^{1}_{0}(\Omega).
\end{equation*}
As a consequence $u = 0$ or $u = u_{0}$ which shows that $u \in U \cup B_{\delta}$. This is again a contradiction. So we finish the proof of this proposition.
\cqd
\end{proof}

Now we able to consider an abstract useful result. First, let $X$ be a reflexive Banach space. Also let $J : X \rightarrow \mathbb{R}$ a functional of class $C^{1}$ which is bounded from below and $\inf J < 0$. Suppose that $J$ satisfies $(Ce)_{c}$ condition for any $c \in \mathbb{R} \backslash \Gamma$ where $\Gamma \subset ( \inf J, 0]$. Under these hypotheses we can prove the following result:
\begin{proposition}\label{v}
Assume also that $\{0, u_{0}\}$ are the only critical points of $J$ where $u_{0}$ is a minimizer of $J$. Suppose that for any neighborhood $U$ of $u_{0}$ and any  $\delta > 0$ such that $U \cap B_{\delta} = \emptyset $, we
can find $\zeta > 0$ satisfying
\begin{equation*}
(1 + \|u\|) \| J^{'}(u)\| \geq \zeta, \forall \,\, u \in J^{c} \backslash ( U \cup B_{\delta}).
\end{equation*}
Then there exists a locally Lipschitz map $v : J^{c}\backslash U \cup B_{\delta} \rightarrow X$ such that
\begin{equation}
       \text{ }
       \left\{ \begin{array}[c]{cc}
           \|v(u)\| \leq 1 + \|u\|, & \\
           \langle J^{'}(u), v(u)\rangle \geq \dfrac{\zeta}{2}, &\\
        \end{array}
      \right.
      \end{equation}
hold for any $u \in J^{c}\backslash U \cup B_{\delta}$.
\end{proposition}
\begin{proof}
The proof is quite standard and we will omit it. We refer the reader to \cite{BN,no}. \cqd
\end{proof}

Now, for a suitable $\widehat{c} > 0$,  we define $D = J^{\widehat{c}}\backslash U \cup B_{\delta}$
where we take
$$int \left(J^{\widehat{c}}\backslash U \cup B_{\delta})\right) \neq \emptyset. $$
In this way, consider $v : D  \rightarrow  H^{1}_{0}(\Omega)$ be the locally Lipschitz map obtained in Proposition \ref{v}.  Therefore the following Cauchy problem:
\begin{equation}\label{cauchy}
       \text{ }
       \left\{ \begin{array}[c]{cc}
           \dfrac{d \eta(t)}{d t} = - \dfrac{v (\eta(t))}{\|v(\eta(t))\|^{2}} \,\, on \,\, [0, \infty) & \\
           \eta(0) = z \\
        \end{array}
      \right.
      \end{equation}
is well defined for any $z \in int \left(J^{\widehat{c}}\backslash U \cup B_{\delta})\right)$. Using the fact that $v$ is locally Lipschitz we know that problem \eqref{cauchy} has a unique local flow.

Next we shall consider the following definition
\begin{definition}\label{li}
Let $X = V \bigoplus W$ be Banach space where $V, W$ are subspaces. Let $J : X \rightarrow \mathbb{R}$ be a functional of class $C^{1}$. We say that $J$ has the linking at the origin when there are $\gamma > 0, \delta > 0$ such that
\begin{description}
  \item[i] $J(u) \geq \gamma, u \in W, \|u\| \leq \delta,$
  \item[ii)] $J(u) \leq 0, u \in V, \|u\| \leq \delta.$
\end{description}
\end{definition}

\begin{remark}
Under the assumptions described in the definition just above it follows that $0$ is a critical point of $J$.
\end{remark}

 Now we will be considered, for easy reference, the Liking Theorem provided by Brezis-Niremberg; see \cite{BN}. This result can be rewritten for strong resonant problems changing the well known Palais-Smale conditon by the Cerami condition for some specific levels of engergy.

\begin{theorem}
 Let $X$ be a reflexive Banach space. Assume that $ X = V \bigoplus W$ where $0 < \dim V < \infty, \, \dim W > 0$. Let  $J : E \rightarrow \mathbb{R}$ be a functional of class $C^{1}$ which is bounded from below such that $\inf J < 0$. Assume that $J$ satisfies $(Ce)_{c}$ condition for any $c \in \mathbb{R} \backslash \Gamma$ where $\Gamma \subset ( \inf J, 0]$. Furthermore, assume that there are $\gamma > 0$ and $\delta > 0$ such that
\begin{description}
  \item[i)]  $J(u) \geq \gamma, u \in W, \|u\| \leq \delta,$
  \item[ii)] $J(u) \leq 0, u \in V, \|u\| \leq \delta.$
\end{description}
Then the functional $J$ has at least two nontrivial critical points. Namely, $J$ has a critical point with negative energy given by minimization and another one with positive energy.
\end{theorem}

The theorem just above is a slightly modification of the usual Theorem 4 in \cite{BN} which explores the Palais-Smale condition. However, since we work with the strong resonance situation, we prefer write this theorem as above-mentioned which is sufficient
 for our problem. The proof of this Theorem follows the same ideas discussed in \cite{BN, no}. We mention that the key for this proof is contained in Propositions \ref{pp}, \ref{v}.  Thus we can find a critical point with positive energy using the fact that $J$ satisfies the $(Ce)_{c}$ condition for any level $c \in \mathbb{R} \backslash \Gamma$.

In that case we will be considered a specific liking geometry due the fact that $J$ presents the linking at the origin quoted in Definition \ref{li}. Let $w_{0} \in W$ be such that $\|w_{0}\| = 1$. Define the set
\begin{equation}
E = \{ u \in H^{1}_{0}(\Omega) : u = \lambda w_{0} +  v, v \in V, \lambda \geq 0, \|u\| \leq 1 \}.
\end{equation}
Clearly, for any $u \in \partial E, u \neq w_{0}$, there are unique $0 \leq \lambda \leq 1, 0 < \mu \leq 1$ and  $v \in V, \|v\| = 1$ such that
$u = \lambda w_{0} + \mu v$.

Using the flow introduced in \eqref{cauchy} we can define a continuous function $p^{\star}: \partial E \rightarrow H^{1}_{0}(\Omega)$ by the following expression
\begin{equation*}
p^{\star}(\lambda w_{0} + \mu v) =
       \text{ }
       \left\{ \begin{array}[c]{cc}
            \eta(2 \lambda \tau(v)), \lambda \in [0, \frac{1}{2}] & \\
           (2\lambda - 1)w_{0} + 2(1 -  \lambda)\eta(\tau(v)), \lambda \in (\frac{1}{2}, 1], &\\
        \end{array}
      \right.
      \end{equation*}
where $\tau : D \rightarrow \mathbb{R}$  is a suitable continuous function.
Now we mention that for any continuous extension $p$ of $p^{\star}$ on all of $E$ we
have
\begin{equation*}
p(E) \cap \partial B_{\rho}^{W} \neq \emptyset
\end{equation*}
for any $\rho > 0$ small where $B_{\rho}^{W} = \{ w \in W : \|w\| = \rho \}$. In other words, $p^{\star}(\partial E)$ and $B_{\rho}^{W}$ link.

Let
$$\Upsilon = \{ p \in C(E, H^{1}_{0}(\Omega)) : p|_{\partial E} = p^{\star} \}$$ and
\begin{equation}\label{c}
c^{\star}= \inf_{p \in \Upsilon} \sup_{u \in E} J(p(u)).
\end{equation}
Note that, by liking geometry at the origin, follows easily that
$$c^{\star} \geq \sup_{u \in E} J(p(u)) \geq \gamma > 0.$$
In particular, we will find a critical point at level $c^{\star}$ when the functional $J$ satisfies the Cerami condition for any $c \in \mathbb{R} \backslash \Gamma$. In particular, $J$ satisfies $(Ce)_{c}$ condition for any $c > 0$.

Now we ready to consider the liking at the origin for our problem. More specifically, we can prove the following result

\begin{proposition}\label{p3}
Suppose $(HSR), (F1), (F4)$. Then the functional $J$ has the following liking at the origin:
\begin{description}
  \item[a)] There are $\eta> 0, \gamma > 0$ such that
  $$J(u) \geq \gamma, \,\,\forall \,\, u \in H_{k}^{\perp} : = span \left\{ \phi_{k + 1}, \ldots \right\} , u \neq 0, \|u\| \leq \eta,$$
  \item[b)] Under these conditions we also obtain that
  $$J(u) \leq 0, \,\, \forall \,\, u \in H_{k} : = span \left\{ \phi_{1}, \ldots, \phi_{k}\right\} , \|u\| \leq \eta.$$
\end{description}
\end{proposition}
\begin{proof}
First of all, by $(F4)$, we easily see that
\begin{equation*}
F(x,t) \leq \dfrac{1}{2} (\lambda_{k + 1} - \lambda_{1} - \epsilon)|t|^{2} + C|t|^{q}, t \in \mathbb{R}, x \in \Omega,
\end{equation*}
for some $q \in (2, 2^{\star})$ and $C > 0$. This implies that
\begin{eqnarray}
  J(u) &\geq& \dfrac{1}{2}\left( 1 - \dfrac{\lambda_{k + 1} - \epsilon}{\lambda_{k + 1}}\right)\|u\|^{2} - C\|u\|^{q} =  \dfrac{\epsilon}{2\lambda_{k + 1}}\|u\|^{2} - C\|u\|^{q} \geq  \dfrac{\epsilon}{4\lambda_{k + 1}}\|u\|^{2}
\end{eqnarray}
for any $u \in H_{k}^{\perp}, \|u\| = \eta_{1}$ where $\eta_{1} > 0$ is small enough.

On the other hand, using one more time $(F4)$, we also see that
\begin{equation*}
F(x,t) \geq \dfrac{1}{2} (\lambda_{k} - \lambda_{1})|t|^{2} , |t| \leq \delta, x \in \Omega.
\end{equation*}
However, using the fact that $H_{k}$ has finite dimension, we obtain that the norms $\|\|$ and $\|\|_{\infty}$ are equivalents on $H_{k}$.
Thus we conclude that
\begin{equation*}
F(x,u) \geq \dfrac{1}{2} (\lambda_{k} - \lambda_{1})|u|^{2},
\end{equation*}
for any $u \in H_{k}, \|u\| = \eta_{2}$ where $\eta_{2} > 0$ is small enough. As a consequence follows that
\begin{equation*}
J(u) = \dfrac{1}{2} \int_{\Omega} |\nabla u|^{2}dx - \dfrac{\lambda_{1}}{2} \int_{\Omega} u^{2} - \int_{\Omega} F(x, u) dx \leq \dfrac{1}{2} \left(1 - \dfrac{\lambda_{k}}{\lambda_{k}}\right)\|u\|^{2} = 0,
\end{equation*}
for any $u \in H_{k}, \|u\| = \eta_{2}$. Then putting  $\eta = \min(\eta_{1}, \eta_{2}) > 0$ follows the proof of this proposition.
\cqd
\end{proof}

\section{The proof of our main Theorems}
In this section we give the proof of our main theorems using useful results proved in the previous section.

\subsection{The proof of Theorem \ref{1}}
Initially, by Lemma \ref{cerami}, we remember that $J$ satisfies $(Ce)_{c}$ condition for any $c \in \mathbb{R}\backslash \Gamma$ where
$$\Gamma = \left\{- \int_{\Omega} F^{+}(x)dx, - \int_{\Omega} F^{-}(x)dx \right\}.$$
Moreover, by Proposition \ref{p1}, $J$ is bounded from below. Using Ekeland's Variational Principle we obtain a critical point $u_{0}$ for $J$ such that
$$J(u_{0}) = \inf\{ J(u) : u \in H^{1}_{0}(\Omega) \} < - \min\left\{\int_{\Omega} F^{+}(x)dx,\int_{\Omega} F^{-}(x)dx \right\} \leq 0.$$
Thus the problem \eqref{pi} possesses at least one solution. This finishes the proof. \cqd

\subsection{The proof of Theorem \ref{2}}
We will divide the proof into two parts. First of all, we consider the proof under hypotheses $(HSR), (F1)$. Note that $u_{0}$ given by the previous theorem is a nontrivial solution because $(F1)$ implies that
$\inf J < 0.$

On the other hand, using Proposition \ref{p2}, by Mountain Pass Theorem we obtain a nontrivial solution $u \in H^{1}_{0}(\Omega)$
such that $J(u) > 0$. Therefore the problem \eqref{pi} admits at least two nontrivial solutions. This completes the proof to the first part.

Now, we will consider the proof for the second part.  Using $(F2)$ we obtain two critical points $u_{1}, u_{2} \in H^{1}_{0}(\Omega)$ as critical point of $J^{\pm}$ defined in Remark \ref{re1}, respectively. In this way, we see also that
$$J^{\pm}(u_{i}) > 0, i = 1, 2.$$
In addition the Maximum Principle implies that
$$u_{1} > 0 \,\,\mbox{and} \,\, u_{2} < 0 \,\, \mbox{in} \,\, \Omega$$
and $u_{1}, u_{2}$ are two distinct nontrivial critical points for $J$.
As a consequence the problem \eqref{pi} possesses at least three nontrivial solutions $u_{0},u_{1}, u_{2}$ where $u_{0}$ was obtained as above-mentioned. This completes the proof. \cqd

\subsection{The proof of Theorem \ref{3}}
Recall that $J$ has at least one critical point $u_{0}$ given by minimization, see Theorem \ref{1}. As a consequence $\inf J = J(u_{0}) < 0$.

On the other hand, by Proposition \ref{p3}, we also obtain a critical point $u_{\star}$ at level $c^{\star} > 0$ which was defined in \eqref{c}. In particular, we see easily that $J(u_{\star}) > 0$.
This shows that $J$ has at least two nontrivial critical points and the problem \eqref{pi} admits at least two nontrivial solutions. So the  proof of this theorem is now complete. \cqd

\begin{remark}
We point out that our main condition $(HSR)$ can be used for other resonance elliptic problems using further variational methods. However, the condition $(Ce)_{c}$ holds only in $\mathbb{R}\backslash \Gamma$. This a serious restriction in strong resonance problems when the main tool is to apply variational methods. So we have a natural ask: Is there critical point for $J$ with energy $c \in \Gamma$. This is a big problem because the functional associate does not satisfy the Cerami condition for those levels. To the best our knowledge these problems are still open.
\end{remark}

\section{Examples}
In this section we will discuss some simple examples where the condition $(HSR)$ is satisfied. After that we shall discuss our main hypotheses in those examples showing the existence and multiplicity of solutions for the problem \eqref{pi}. Initially, we consider

$\mathbf{Example}$ 1: Let $g: \mathbb{R} \rightarrow \mathbb{R}$ be a function of class $C^{2}$ such that
$$\lim_{|t| \rightarrow \infty} \dfrac{g^{'}(t)}{t} = 0 \,\,\mbox{and} \,\, g^{'}(0)= 0.$$
Also let $a: \Omega \rightarrow \mathbb{R}$ be a continuous function.
Then
$$F(x,t) = \dfrac{a(x) g(t)}{1 + t^{2}}, x \in \Omega, t \in \mathbb{R}$$
satisfies
$$\lim_{|t| \rightarrow \infty} F(x,t) = 0.$$
Besides that, taking $f(x,t) = F^{'}(x,t)$, we obtain that
$$\lim_{|t| \rightarrow \infty} t f(x, t) = \lim_{|t| \rightarrow \infty} f(x, t) = 0.$$
In particular, the function $f$ verifies $(HSR)$. In that case the functional $J$ satisfies $(Ce)_{c}$ condition if only if $c \neq 0.$

Next we will assume that $a \equiv 1$. Under this hypothesis, using the L'Hospital rule, we see that
$$\lim_{|t| \rightarrow 0} \dfrac{2 F(x,t)}{t^{2}} = g^{''}(0).$$
Thus, assuming also that $g^{''}(0) > 0$, follows that $(F1)$ is verified. Here we mention the fact that $u = 0$ is a trivial solution for the problem \eqref{pi} because $f(x,0) \equiv 0$. As a consequence the problem \eqref{pi} possesses at least one nontrivial solution given by Theorem \ref{2}.

Next will consider another example where $f(x,t)$ is equal to zero for $|t|$ big enough. More specifically, we consider

$\mathbf{Example}$ 2: Let $h:\mathbb{R} \rightarrow \mathbb{R}$ be a continuous function such that
$$h(t) = 0, \,\, \mbox{for any} \,\, |t| \geq l$$
where $l > 0$. Then
$$f(t) = t h(t), t \in \mathbb{R}$$
satisfies $(HSR)$ and the $(Ce)_{c}$ condition for any $c \in \mathbb{R}\backslash \Gamma$
where we define
$$\Gamma = \left\{- F^{+}|\Omega|, - F^{-}|\Omega| \right\}$$
and
$$F^{+}= \int_{0}^{l} sh(s)ds, F^{-} = - \int_{-l}^{0} sh(s)ds.$$
Supposing also that
$$\int_{0}^{l} sh(s)ds = \int_{-l}^{0} sh(s)ds = 0$$
holds it follows that the associated functional $J$ satisfies the $(Ce)_{c}$ if only if $c \neq 0$. In addition, assuming also that
$h(0) > 0$ holds follows that $(F1)$ is verified. Indeed, we see that
$$\lim_{t \rightarrow 0} \dfrac{2 F(t)}{t^{2}} = \lim_{t \rightarrow 0} \dfrac{ f(t)}{t} = h(0) > 0.$$
This fact shows that
$$F(t) \geq h(0)\dfrac{t^{2}}{2}, |t| \leq \delta,$$
for some $\delta > 0$ small. As a consequence
$$\int_{\Omega} F(t \phi_{1})dx \geq \dfrac{h(0) t^{2}}{2} \int_{\Omega} \phi_{1}^{2}dx = \dfrac{h(0) t^{2}}{2} > 0,$$
for any $t > 0$ small enough. Then the problem \eqref{pi} admits at least one solution provided by Theorem \ref{2}. Actually we can also see that $(F3)$ is verified. Therefore the problem \eqref{pi} has at least two nontrivial solutions given by Theorem \ref{2}.

When
$$h(0) \in \left(\lambda_{k} - \lambda_{1}, \lambda_{k + 1} - \lambda_{1} - \epsilon \right), k \geq 2,$$
holds for some $\epsilon > 0$ it follows easily that $(F4)$ is satisfied. In that case, assuming $(F5)$, the problem \eqref{pi} has at least two nontrivial solutions given by Theorem \ref{3} one of them with negative energy and another one with positive energy.

Now we shall consider the following example:

$\mathbf{Example}$ 3: Define the function $F:\mathbb{R} \rightarrow \mathbb{R}$ by
$$F(t) = \dfrac{ln(1 + t^{2})}{1 + t^{2}}, t \in \mathbb{R}.$$
We easily see that $f(t) = F^{'}(t)$ satisfies $(HSR)$ and $(Ce)_{c}$
condition if only if $c \neq 0$. Besides that, $f(0) = 0$ and the hypothesis (F1) is trivially verified. Therefore the Theorem \ref{1} give us a nontrivial solution for the problem \eqref{pi}. Moreover, using Theorem \ref{2}, follows that the problem \eqref{pi} admits at least two nontrivial solutions.

$\mathbf{Acknowledgments}$: The author was partially supported by CNPq-Procad-UFG with grant number: \\

\end{document}